\documentclass[12pt]{amsart}
\usepackage{amsfonts}
\usepackage{amsmath, amsthm, amssymb,color}
\usepackage{latexsym}
\usepackage{txfonts}
\usepackage{bm}
\usepackage{graphicx}
\usepackage{graphics}
\usepackage{epsfig}
\usepackage[polish,english]{babel}
\usepackage[T1]{fontenc}
\usepackage{mathrsfs}

\numberwithin{equation}{section}
\allowdisplaybreaks

\newtheorem{lemma}{Lemma}[section]

\newtheorem{theorem}[lemma]{Theorem}
\newtheorem{proposition}[lemma]{Proposition}
\newtheorem{definition}[lemma]{Definition}
\newtheorem{corollary}[lemma]{Corollary}
\newtheorem{example}[lemma]{Example}
\newtheorem{exercise}[lemma]{Exercise}
\newtheorem{remark}[lemma]{Remark}
\newtheorem{fig}[lemma]{Figure}
\newtheorem{tab}[lemma]{Table}

\newcommand{\bth}{\begin{theorem}}
\newcommand{\ethe}{\end{theorem}}
\newcommand{\bre}{\begin{remark}\em }
\newcommand{\ere}{\end{remark}}
\newcommand{\ble}{\begin{lemma}}
\newcommand{\ele}{\end{lemma}}
\newcommand{\bde}{\begin{definition}}
\newcommand{\ede}{\end{definition}}
\newcommand{\bco}{\begin{corollary}}
\newcommand{\eco}{\end{corollary}}
\newcommand{\bpr}{\begin{proposition}}
\newcommand{\epr}{\end{proposition}}

\newcommand{\bexer}{\begin{exercise}}
\newcommand{\eexer}{\end{exercise}}

\newcommand{\bexam}{\begin{example}\rm  }
\newcommand{\eexam}{ \end{example}}

\newcommand{\bfi}{\begin{fig}}
\newcommand{\efi}{\end{fig}}

\newcommand{\btab}{\begin{tab}}
\newcommand{\etab}{\end{tab}}

\def\E{{\mathbb{E}}}

\def\R{{\mathbb{R}}}

\def\N{{\mathbb{N}}}

\def\B_e{B_{\eta}(e)}

\newcommand{\ov}{\overline}

\definecolor{darkblue}{rgb}{0,0,1}
\definecolor{darkgreen}{rgb}{0,1,0}
\definecolor{darkred}{rgb}{1, 0,0}

\newcommand{\sign}{{\rm sign}}

\newcommand{\beao}{\begin{eqnarray*}}
\newcommand{\eeao}{\end{eqnarray*}\noindent}

\newcommand{\beam}{\begin{eqnarray}}
\newcommand{\eeam}{\end{eqnarray}\noindent}

\newcommand{\beqq}{\begin{equation}}
\newcommand{\eeqq}{\end{equation}\noindent}

\newcommand{\bce}{\begin{center}}
\newcommand{\ece}{\end{center}}

\newcommand{\barr}{\begin{array}}
\newcommand{\earr}{\end{array}}

\newcommand{\vague}{\stackrel{\lower0.2ex\hbox{$\scriptscriptstyle
                    \it{v} $}}{\rightarrow}}
\newcommand{\weak}{\stackrel{\lower0.2ex\hbox{$\scriptscriptstyle
                    \it{w} $}}{\rightarrow}}
\newcommand{\what}{\stackrel{\lower0.2ex\hbox{$\scriptscriptstyle
                    \it{\hat{w}} $}}{\rightarrow}}

\newcommand{\bdis}{\begin{displaymath}}
\newcommand{\edis}{\end{displaymath}\noindent}

\newcommand{\bali}{\begin{align}}
\newcommand{\eali}{\end{align}}

\usepackage{ascmac}	
\begin{document}

\bibliographystyle{alpha}
\title[Asymptotics of MLE for stable law with continuous parameterization]
      {Asymptotics of maximum likelihood estimation for stable law with continuous parameterization}
\today
\author[M. Matsui]{Muneya Matsui}
\address{Department of Business Administration, Nanzan University,
18 Yamazato-cho Showa-ku Nagoya, 466-8673, Japan}
\email{mmuneya@nanzan-u.ac.jp}

\begin{abstract}
 Asymptotics of maximum likelihood estimation for $\alpha$-stable law
 are analytically investigated with a continuous parameterization. 
 The consistency and asymptotic normality are shown on the interior of
 the whole parameter space. Although these asymptotics have been
 provided with Zolotarev's $(B)$ parameterization, there are several gaps between. 
 Especially in the latter, the density, so that scores and their
 derivatives are discontinuous at $\alpha=1$ for $\beta\neq 0$
 and usual asymptotics are impossible. This is considerable
 inconvenience for applications. By showing that these quantities are
 smooth in the continuous form, we fill gaps between and provide a
 convenient theory. 
 We numerically approximate the Fisher information matrix around the Cauchy
 law $(\alpha,\beta)=(1,0)$. The results exhibit  
 continuity at $\alpha=1,\,\beta\neq 0$ and this secures the accuracy of
 our calculations. 
\vspace{2mm}  
\end{abstract}
\thanks{Muneya Matsui's research is partly supported by the JSPS Grant-in-Aid
for Young Scientists B (16k16023).
}
\maketitle 

\section{Introduction}
Stable distributions constitute a class of limit distributions of
generalized central limit theorem, including the normal distribution on
the border. Except for Gaussian
, they do not have the second
moment and the class is crowned as a representative of heavy tailed
distributions. 
Moreover, they allow skewness and changes in supports depending on
parameters. 
Due to such a variety of characteristics, they play crucial roles in
both theory and applications. Many statistical models adopt stable
random variables (r.v.'s for short) for their random components.
However there is a well-known bottleneck in applications, namely most
stable laws have no closed form density functions and only
their characteristic functions (ch.f.'s) are explicit. Thus we need to
devise methods whenever the stable laws are applied. 
For more details and other notable properties, consult, 
e.g. \cite{samorodnitsky:taqqu:1994} and references therein. 

In applications, several parameterizations ($A,B,C,E$
and $M$ by \cite{zolotarev} and their variants by \cite{nolan:1998}) are
available in terms of the ch.f. 
They have both strong and weak points in each. We leave detailed explanations to the references
\cite{zolotarev,samorodnitsky:taqqu:1994,nolan:1998,nolan::2018}.
Our focus here is on a continuous parameterization, which is a 
modified version of $(M)$ by \cite{nolan:1998}\footnote{At
$\alpha=1,\,\beta\neq 0$, 
Zolotarev's original $(M)$ parameterization is not a location-scale
family and Nolan considered a modified version called $(0)$
parameterization. We reconcile them and use $M_0$
parameterization here. 
}. We call it $(M_0)$
parameterization. This has desirable properties in statistical applications. 
It is a location-scale
family, and moreover,
its characteristic function (ch.f.) is 
continuous with respect to all stable parameters
$(\mu,\sigma,\alpha,\beta)$, 
so that we could treat
the distribution continuously in 
the whole parameter
space. 
Several papers recommend to use the representation for 
statistical applications (see, e.g. \cite{nolan:1998} and \cite[Chap.I]{nolan::2018}).

Concerning the maximum likelihood estimation (MLE) for stable laws, 
DuMouchel has theoretically investigated the asymptotics with $(B)$ expression
(\cite{dumouchel:1973}) and 
calculated the Fisher information for $(A)$ form (\cite{dumouchel:1975})
\footnote{To be accurate, the definition of $(B)$ in \cite{dumouchel:1973}
is slightly different from our version of $(B)$ by \cite{zolotarev}. However, since 
we could simply imitate the theoretical approach in \cite{dumouchel:1973} for
the theory of our version, we do not distinguish the two forms here. See
Section \ref{sec:discussion} for more details}. 
Unless parameters are in the neighborhoods of $(\alpha,\beta)=(1,0)$ or the 
boundaries, the asymptotics for $(B)$ are easily converted to those for
$(A)$. Indeed numerical results \cite{dumouchel:1975} of $(A)$ are based on the
theory in \cite{dumouchel:1973}. 
For this reason we could say that theoretical studies 
are sufficient for applications of $(A)$ and $(B)$ types. 

However, as far as we know, there are no concrete asymptotic theories
for MLE with $(M_0)$ form, though possibility is suggested in
\cite{dumouchel:1973}. Even when the $(M_0)$ type stable law was used, only 
DuMouchel \cite{dumouchel:1973} has been referred (see, e.g. \cite{nolan:2001}
or \cite{andrews:calder:davis:2009}, we also personally communicated with John Nolan). 



In this paper we analyze the asymptotics of 
MLE for $(M_0)$ parameterization. 
More precisely, we present
the consistency and asymptotic normality of MLE. Our main tools for
deriving asymptotic are the detailed analysis of score functions and
their derivatives. We rigorously show that the scores so that the Fisher informations are
continuous at $\alpha=1$. The difficulty there is that the score
functions include multiple diverging terms, which are proved to be
canceled each other out. 
Since the case $\alpha=1,\beta\neq 0$ is excluded for $(A)$ and $(B)$ 
types due to discontinuity, the obtained results contrast with
established asymptotics by \cite{dumouchel:1973}.

In derivation of the asymptotics, we have to start with properties of the $(M_0)$ density and
its derivatives, since the previous investigation has been done with $(B)$
type, which has a very convenient ch.f. form for the density analysis (see
\cite{zolotarev}). 
We go back to the ch.f. for $(M_0)$, from which we derive
necessary properties of the density for the asymptotics of MLE.
In a part of the process, we effectively use a relation of $(M_0)$ and $(B)$ on
possible parameter regions\footnote{Notice that the $(B)$ type in \cite{dumouchel:1973}
is different from our $(B)$ form and thus even for our $(B)$ type, we need to derive necessary properties
for the asymptotics separately.}. 
Our theoretical base is a rather modern
and sophisticated one \cite{vandervaat:2000}, which is relatively easy
to handle. Therefore the established 
theories could be developed and arranged in various ways for related applications.  

Preliminary results are obtained in Section \ref{sec:preliminary}.
For $\alpha\neq 1$ the  
tail behaviors of derivatives with $(M_0)$ form are derived 
via those of $(B)$ form, whereas at $\alpha=1$, these quantities are 
independently derived. A combination of these preliminary results
constitutes the tail behaviors of scores (Proposition \ref{prop:scores}). 
The consistency and asymptotic normality are presented in Section
\ref{sec:consan}, which are our main results. 
In order to check the continuity of the Fisher information around $\alpha=1$, a small numerical
work is conducted in Section \ref{sec:fic}. 
We discuss new and
known things in the literature in Section \ref{sec:discussion}, where
several future works are suggested. 



\section{Preliminary results}
\label{sec:preliminary}
This preliminary starts with characteristic function (ch.f. for short) of $(M_0)$ parameterization and its several
properties. Then we proceed to the tail behaviors of $(M_0)$ 
density and its derivatives (Lemma \ref{lem:ff}), which are combined for analyzing the tail behaviors of
score functions (Proposition \ref{prop:scores}).

Denote ch.f. of $(M_0)$ parameterization by 
\begin{align}
\label{chfM}
\begin{split}
 \varphi(t) 
= \left \{
\begin{array}{ll}
\exp\Big(
 -|\sigma t|^\alpha \big\{
 1+i\beta\, \sign t \tan \frac{\pi\alpha}{2}(|\sigma t|^{1-\alpha}-1)
\big\} +i\mu t
\Big)  & \text{if}\ \alpha\neq 1\\
 \exp\Big(
-|\sigma t| -i  \sigma t \, (2\beta /\pi) \log |\sigma t| +i\mu t
\Big)  & \text{if}\ \alpha= 1,
\end{array}
\right.
\end{split}
\end{align}
where $\mu\in \R,\,\sigma \in \R_+\,, \alpha\in(0,2]\,,\beta\in[-1,1]$ with
$\R_+= (0,\infty)$. We denote this parameter space by $\Theta_{M}$ and its
interior by $\Theta_{M}^\circ$. 
The expression \eqref{chfM} shows continuity in $\alpha$. 
We see that the
density $f$ is a location-scale family. 
Indeed the inversion formula for $\alpha\neq 1$ yields
\begin{align}
\label{fdensity:m}
 f(x;\mu,\sigma,\alpha,\beta)&
= \frac{1}{\pi} \mathrm{Re}
 \int_0^\infty e^{-it \big(x-\mu+\sigma \beta \tan
 \frac{\pi\alpha}{2} \big)-|\sigma t|^\alpha
\big(1-i\beta\tan \frac{\pi\alpha}{2} \big)
}dt \\
&=\frac{1}{\sigma} \frac{1}{\pi} \mathrm{Re}
 \int_0^\infty e^{-it \big(\frac{x-\mu}{\sigma}+\beta \tan
 \frac{\pi\alpha}{2}\big)-|t|^\alpha
\big(1-i\beta\tan \frac{\pi\alpha}{2} \big)
}dt \nonumber\\
&= \frac{1}{\sigma} f(\frac{x-\mu}{\sigma};0,1,\alpha,\beta).   
\nonumber
\end{align}
In a similar way or by continuity, we can check it also at $\alpha=1$.  

We use the following notations throughout. 
A parameter vector is denoted by  
$\theta=(\theta_1,\theta_2,\theta_3,\theta_4)'=(\mu,\sigma,\alpha,\beta)'$. 
As usual $f',f''$ mean the first and the second derivatives with respect
to (w.r.t.) $x$ and
$f_\theta=(f_{\theta_1},f_{\theta_2},f_{\theta_3},f_{\theta_4})'$
denotes a vector of partial derivatives of $f$ w.r.t. 
$\theta$. The second order partial derivatives w.r.t. $x$ and $\theta$
are denoted by  
\[
 f_{\theta_i}'=\frac{\partial^2 f}{\partial x \partial \theta_i
      }=\frac{\partial^2 f}{\partial \theta_i \partial x },\quad
 f_{\theta_i\theta_j}= \frac{\partial^2 f}{\partial \theta_i \partial
      \theta_j} = \frac{\partial^2 f}{\partial \theta_j \partial
      \theta_i},\quad 
      i,j=1,\ldots,4,  
\]
i.e. all derivatives will be shown to be interchangeable in our case. 
Moreover, $\varphi_{\theta_i},\,\varphi_{\theta_i\theta_j}$ respectively
denote the first and the second derivatives of $\varphi$.
As is well known,
$\varphi_{\theta_i},\,\varphi_{\theta_i\theta_j}$ are represented by 
those of cumulant $\psi(t)=\log \varphi(t)$: $\varphi_{\theta_i}=\psi_{\theta_i}\varphi$ and
      $\varphi_{\theta_i\theta_j}=(\psi_{\theta_i\theta_j}+\psi_{\theta_j}
\psi_{\theta_j})\,\varphi$, 
where 
$\psi_{\theta_i},\,\psi_{\theta_i\theta_j}$ are derivatives of $\psi$. 

\subsection{Behavior of derivatives of density $f$ 
  w.r.t. $\theta$ and $x$}
Here we check continuous differentiablity of $f$ and 
obtain tail bounds for derivatives. 
In the derivation of bounds, 
we use a relation between $(M_0)$ and $(B)$ forms, which are possible 
on a restricted parameter space.
We obtain tail bounds in $(B)$ form (Appendix
\ref{subsec:rela:mb}) and exploit them for 
finding bounds in $(M_0)$ form. 
If we could not use the relation, we directly obtain bounds from
derivatives of ch.f. of $(M_0)$, 
which is done via the inversion formula. Recall that $\Theta_M$ is our
parameter space and $\Theta_M^\circ$ is its interior. 
\begin{lemma}
\label{lem:ff}
For every $x\in \R$, $f(x;\theta):\theta\in \Theta_M^\circ$ is twice
 continuously differentiable w.r.t. $\theta$, and
 $f_{\theta_i},\,i=1,\ldots,4$ is continuously differentiable
 w.r.t. $x$. Moreover
 $f_{\theta_i},\,f_{\theta_i}',\,f_{\theta_i\theta_j}\,i,j=1,\ldots,4$ are jointly
 continuous in $(x,\theta)$ on $\R \times \Theta_M^\circ$. 

The tails of $f$ and its derivatives for sufficiently large $|x|$ satisfy
\begin{align}
\label{tail:derivativesff}
\begin{split}
\begin{array}{lll}
f=O(|x|^{-(1+\alpha)}), & f_\mu'=-f_{\mu\mu}=O(|x|^{-(3+\alpha)}), &  f_{\sigma\alpha}=O(|x|^{-(1+\alpha)} \log |x|), \\
f_\mu =-f'=O(|x|^{-(2+\alpha)}), &  f_\sigma'=-f_{\mu\sigma}=O(|x|^{-(2+\alpha)}), &
  f_{\sigma\beta} = O(|x|^{-(1+\alpha)}),\\
f_\sigma= O(|x|^{-(1+\alpha)}), & f_\alpha'=-f_{\mu\alpha}=O(|x|^{-(2+\alpha)} \log |x|), &
f_{\alpha\alpha}= O(|x|^{-(1+\alpha)} \log^2
  |x|), \\
 f_\alpha =O(|x|^{-(1+\alpha)} \log |x|), & f_\beta'=-f_{\mu\beta}= O(|x|^{-(2+\alpha)}), & f_{\alpha\beta}=O(|x|^{-(1+\alpha)} \log |x|), \\
 f_\beta = O(|x|^{-(1+\alpha)}), &
  f_{\sigma\sigma}=O(|x|^{-(1+\alpha)}), &f_{\beta\beta}=
  O(|x|^{-(1+\alpha)}). 
\end{array}
\end{split}
\end{align}
Furthermore for $\alpha=1,\,\beta\in (-1,1)$, we have  
\begin{align}
 f_{\sigma\sigma}=O(|x|^{-3}\log |x|),\quad
 f_{\beta\beta}=O(|x|^{-3}\log |x|).   \label{asexp:alpha=1}
\end{align} 
\end{lemma}

Notice that these orders are upper bounds and could possibly be smaller
depending on parameter regions. For instance, in the symmetric case
$\beta=0$, we could obtain better orders. 

\begin{proof}
First we assume $\alpha\neq 1$. For twice continuous differentiability
 of $f$, we observe derivatives of the inversion form 
\begin{align}
\label{secondderiv:f}
 f_{\theta_i\theta_j}(x;\theta)=\frac{1}{2\pi} \int_{-\infty}^\infty
 e^{-it x}\varphi_{\theta_i \theta_j}(t) dt,\quad 1\le i,j\le 4,  
\end{align}
where differentiations are done under the integral sign. 
Indeed, since $\varphi_{\theta_i\theta_j}$ is constructed with
 $e^{-|t\sigma|^\alpha}$ multiplied by a linear combination of powers of
 $|t|$ and $\log |t|$, the absolute values of integrands are
 integrable regardless of value of $x\in \R$ (see
 Lemma \ref{lem:derivatives:logchfs} for exact forms of
 $\psi_{\theta_i\theta_j}$, 
 so that $\varphi_{\theta_i \theta_j}$). 
 In the form \eqref{secondderiv:f} it is not difficult to see continuity of
 $f_{\theta_i \theta_j}$ in $(\theta_i,\theta_j)$ by the dominated
 convergence theorem (DCT for abbreviation).
 
 For differentiability of $f_{\theta_j}$
 w.r.t. $x$, similarly as before, it suffices to look definability
 and continuity of forms  
\begin{align}
\label{firstderivtx:f}
 f_{\theta_i}'(x;\theta)=\frac{1}{2\pi} \int_{-\infty}^\infty
 it e^{-it x}\varphi_{\theta_i}(t) dt, \quad 1\le i \le 4, 
\end{align}
where $\varphi_{\theta_i}$ contains the term $e^{-|t\sigma|^\alpha}$ and
 each integrand of right side is absolutely integrable. Since integrands
 of \eqref{firstderivtx:f} are continuous in $x$, the result follows
 from DCT. 

We take up $f_{\theta_i}$ and show 
 continuity in $(x,\theta)\in \R \times \Theta_M^\circ$.
 With another point $(y,\theta')\in \R\times \Theta_M^\circ$,
 we write 
\[
 f_{\theta_i}(x;\theta)-f_{\theta_i}(y;\theta')=\frac{1}{2\pi}
 \int_{-\infty}^\infty \big\{ e^{-itx}(1-e^{it(x-y)})\varphi_\theta(t)+e^{-it
 y}(\varphi_\theta(t)-\varphi_{\theta'}(t)) \big\} dt,
\]
where an inequality
 $|1-e^{it(x-y)}|\le c |t(x-y)|^\gamma, 0<\gamma\le 2,\,c>0$ yields 
a dominant function of the first part, while the second part is continuous
 in $\theta$ regardless of $y$. 
Thus again by DCT 
we obtain joint continuity. 
To make sure, we
 present the exact forms of $\psi_{\theta_i},\psi_{\theta_i\theta_j}$ so that $\varphi_{\theta_i},\varphi_{\theta_i\theta_j}$
 in Lemma \ref{lem:derivatives:logchfs}. We omit the proof for
 $f_{\theta_i}',\,f_{\theta_i\theta_j}$ which is similar. 

 When $\alpha=1$, we need a special treatment, since as $\alpha\to 1$
 several terms in $\varphi_{\theta_i},\varphi_{\theta_i\theta_j}$ are
 diverging to $\infty$, which are proved to be canceled one another in the
 end. This is done in Lemma \ref{lem:derivatives:logchfs}, where we could see joint continuity
 of $\varphi_{\theta_i}(t),\,t\varphi_{\theta_i}(t)$ and
 $\varphi_{\theta_i\theta_j}(t)$ in $(\theta,t) \in \Theta_M^\circ\times \R$. Moreover, all of these
 quantities, as functions of $t$, have dominant integrable
 functions (Lemma \ref{lem:dervcauch}). Therefore, we can reuse the proof in case $\alpha\neq 1$
 also for $\alpha=1$. We omit further details. 

Next we proceed to the tail bounds. 
We start with the case $\alpha\neq 1$. Since there is the relation
 between 
$(M_0)$ and $(B)$ forms (Lemma \ref{lem:a1}), it suffices to use the tail
 bounds of $(B)$ in Lemma \ref{lem:derivg}. Namely we choose maximum
 tail functions among \eqref{deriveggg} in the expressions \eqref{derivf:g1}. 

When $\alpha=1$ the proof is more complicated and we only state the 
 outline taking up $f_{\theta_i}$. Proofs for other quantities
 $f_{\theta_i}',f_{\theta_i\theta_j}$ are similar. 
We notice that 
characteristic functions of $(M_0)$ form 
$\varphi$ and $(B)$ form $\varphi_B$ (see \eqref{chfM} and \eqref{ch.f.Bform}) differ only in the
 constant $\pi/2$ of scale parameter, so that $\varphi$ could be
 analytically extended to the complex plane, which is done for
 $\varphi_B$ in \cite[Ch.2]{zolotarev}. We also apply this extension to
$e^{-itx}\varphi_{\theta_i}(t)$ and consider the
 contour integration as in \cite[Theorem 2.5.4]{zolotarev}\footnote{Notice
 that there are several flows in the proof of Theorem 2.5.4 in \cite{zolotarev}, however,
 we check that the method is correct.}. Then 
for $x>0$ and $\beta \in (-1,1)$ we have 
\begin{align}
\label{eq:fa1theta}
f_{\theta_i}(x) 
&= \mathrm{Re} \frac{1}{\pi}\int_0^\infty \psi_{\theta_i}(t) \,e^{-itx -t -i2\beta/\pi t\log
 t}dt 
\quad \left( = 
\mathrm{Re} \frac{1}{\pi} \int_0^\infty e^{-itx} \varphi_{\theta_i}(t)
 dt 
\right)
\\
&= \mathrm{Im} \frac{1}{\pi} \int_0^\infty \psi_{\theta_i}(t/i)
 \,e^{-tx+i(1+\beta)t-2\beta /\pi  t\log t}dt \nonumber \\  
&= \mathrm{Im}\frac{1}{\pi x}  \int_0^\infty \psi_{\theta_i}(t/(ix))
 \,e^{-t+i(t/x) (1+\beta)-2\beta/\pi (t/x) \log (t/x)} dt . \nonumber
\end{align} 
Now due to Taylor expansion of 
\[
 \exp \Big( i\frac{t}{x}(1+\beta)-2 \frac{\beta}{\pi} \frac{t}{x} \log
 \frac{t}{x}\Big)
= \sum_{k=0}^\infty \frac{1}{k!} \Big\{
i \frac{t}{x} (1+\beta) -\frac{2\beta}{\pi} \frac{t}{x} \log \frac{t}{x}
\Big\}^{k},
\] we could take the dominant term as $x\to\infty$ and obtain the tail behaviors. The exact forms
 of $\psi_{\theta_i},\,\psi_{\theta_i\theta_j}$ are given in Lemma
 \ref{lem:derivatives:logchfs}. 
In the case $x<0$, we use the relation
\begin{align*}
 f_{\theta_i}(x) &= \mathrm{Re} \frac{1}{\pi}\int_0^\infty e^{itx} \ov
 \varphi_{\theta_i}(t) dt = \mathrm{Re} \frac{1}{\pi}\int_0^\infty \ov
 \psi_{\theta_i}(t) e^{-it(-x)-t-i2(-\beta)/\pi t \log t}dt,  
\end{align*}
where we take the complex conjugate of \eqref{eq:fa1theta}. Now replacing $\psi_{\theta_i}$ by $\ov \psi_{\theta_i}$ and $\beta$ by
 $-\beta$ in \eqref{eq:fa1theta}, we could apply the former method. 
\end{proof}

\begin{remark}
 Results in Lemma \ref{lem:ff} might be obtained directly from
 derivatives of the inversion form for $(M_0)$ expression \eqref{fdensity:m}. 
 Namely, after partially differentiate ch.f. $\varphi$ of \eqref{chfM}
 we could apply asymptotic expansions to the inversion formula 
 of partial derivatives. However these expansions might require a systematic treatment
 of complex contour integrals as done in \cite{zolotarev} with $(B)$
 expression, which is quite long. This is a challenge for future.   
\end{remark}

\subsection{Behavior of score functions and their derivatives
  w.r.t. $\theta$ and $x$}
\label{sec:scores}
Now we study score functions and their derivatives. 
Denote the log-likelihood function
of $f$ and its scores respectively by 
\begin{align}
\label{def:score:M}
 \ell (x;\theta) = \log f(x;\theta)\quad \mathrm{and} \quad
 \ell_{\theta_i}(x;\theta) = \frac{\partial \ell(x;\theta)}{\partial
 \theta_i} = \frac{f_{\theta_i}(x;\theta)}{f(x;\theta)}, \quad 1 \le
 i\le 4, 
\end{align}
where for convenience we sometimes write $\ell(x)$ and
$\ell_{\theta_i}(x)$ for these quantities. 
The second order derivatives of score functions w.r.t. $\theta$ and $x$, denoted by  
\begin{align}
 \label{def:derivscore:M}
 \ell_{\theta_i}'(x)&=\frac{\partial \ell_{\theta_i}(x;\theta)}{\partial x} =
 \frac{1}{f^2(x;\theta)}\big( f_{\theta_i}'(x;\theta)f(x;\theta)-f_{\theta_i}(x;\theta)f'(x;\theta)
\big),  \\
 \ell_{\theta_i\theta_j}(x) &= \frac{\partial
 \ell_{\theta_i}(x;\theta)}{\partial \theta_j} =
 \frac{1}{f^2(x;\theta)}\big( f_{\theta_i\theta_j}(x;\theta)f(x;\theta)
-f_{\theta_i}(x;\theta)f_{\theta_j}(x;\theta)
\big), \label{def:derivscorepara:M} \quad 1\le i,j\le 4,  
\end{align}
are also investigated. 
Here orders of partial derivatives w.r.t. $(x,\theta)$ are all exchangeable. 
These quantities are inevitable for statistical
applications other than the proof of asymptotics of MLE   
such as statistics where estimated parameters are inserted. 
We rigorously show the definability and properties of
 \eqref{def:score:M} - \eqref{def:derivscorepara:M}. 
\begin{proposition}
\label{prop:scores}
Let $\theta \in \Theta_M^\circ$. 
For every $x\in \R$, 
\begin{align}
\label{score:M}
 \ell_{\theta_i}(x),\,\ell_{\theta_i}'(x) \quad \text{and}\quad
 \ell_{\theta_i\theta_j}(x), \quad i,j=1,\ldots,4, 
\end{align}
 are well-defined and continuous in $\theta$, 
 and they are jointly continuous in $(x,\theta)$ on $\R \times \Theta_M^\circ$.   
 Concerning tail behaviors, we have for sufficiently large $|x|,\,x\in \R$,
\begin{align}
 \begin{split}
 \label{tail:derivscore:M}
 \begin{array}{ll}
 \ell_\mu (x) =O(|x|^{-1}), & \ell_{\mu\mu}(x)=-\ell_\mu'(x) =O(|x|^{-2}), \\
 \ell_\sigma (x) =O(1), & \ell_{\mu\sigma}(x)=-\ell_\sigma'(x) =O(|x|^{-1}), \\
 \ell_\alpha (x)=O(\log |x|), & \ell_{\mu\alpha}(x)= -\ell_\alpha'(x) =O(|x|^{-1}\log |x|), \\
 \ell_\beta (x)=O(1), & \ell_{\mu\beta}(x)=-\ell_\beta'(x) =O(|x|^{-1}), 
 \end{array}
 \end{split}
\end{align}
and moreover, 
\begin{align}
\label{tail:secondderivscore:M}
\begin{split}
\begin{array}{lll}
 \ell_{\sigma\sigma}(x)=O(1), & \ell_{\sigma\alpha}(x)=O(\log |x|), &
  \ell_{\sigma\beta}(x)=O(1), \\
 \ell_{\alpha\alpha}(x)=O(\log^2 |x|), &
  \ell_{\alpha\beta}(x)=O(\log |x|),& \ell_{\beta\beta}(x)=O(1). 
\end{array}
\end{split}
\end{align}
\end{proposition}
 Notice that results \eqref{tail:derivscore:M} and
 \eqref{tail:secondderivscore:M} are upper bounds, so that we could obtain
 sharper results depending on parameter values. For example, in the symmetric
 case $(\beta=0)$ the results are more explicit (see \cite{matsui:takemura:2006}). 

\begin{proof}
 The proof for properties of \eqref{score:M} follows from Lemma
 \ref{lem:ff} together with definitions
 \eqref{def:score:M} - \eqref{def:derivscorepara:M}. Notice that 
 $\alpha$-stable distributions are unimodal and for every $x\in\R$, 
 $f(x;\theta)\neq 0$ on $\theta\in \Theta_M^\circ$, and thus continuity of
 $f,\,f_{\theta_i},\,f_{\theta_i}',\,f_{\theta_i\theta_j}$
 in $\theta \in \Theta_M^\circ$ yields that of scores and their
 derivatives \eqref{score:M}.   
 Moreover, $f,\,f_{\theta_i},\,f_{\theta_i}',\,f_{\theta_i\theta_j}$ are
 jointly continuous in $(x,\theta)\in \R\times \Theta_M^\circ$, so that the joint
 continuity of \eqref{score:M} follows. 
 
 Next we prove \eqref{tail:derivscore:M} and
 \eqref{tail:secondderivscore:M}. Notice that the tail order of $f$ in
 Lemma \ref{lem:ff} is exact, i.e. $f(x) \sim c|x|^{-1-\alpha},\,c>0$ as $|x|\to\infty.$
  By substituting the result
 \eqref{tail:derivativesff} of Lemma
 \ref{lem:ff} into the definitions
 \eqref{def:score:M} - \eqref{def:derivscorepara:M}, we could bound them
 from the upper side for sufficiently large $|x|$. Thus we could easily
 reach \eqref{tail:derivscore:M} and
 \eqref{tail:secondderivscore:M}.
\end{proof}

\section{Consistency and asymptotics normality of MLE}
\label{sec:consan}
For asymptotics of maximum likelihood estimate, we rely on a series of theorems in
\cite{vandervaat:2000} adopted to our present situation.
We are starting with additional notations. 
Let $P_\theta:\theta \in \Theta_M$ denote the 
probability measure of stable law with $(M_0)$ parameterization. 
Let $X_1,X_2,\ldots,X_n$ be an iid sample from $(P_\theta:\theta\in
\Theta_M)$ with generic r.v. $X$. 
The log likelihood based on $n$ samples 
is given by   
\begin{align}
\theta \mapsto L_n(\theta) = \frac{1}{n} \sum_{k=1}^n 
 \ell(X_k;\theta),
\end{align}
so that its expectation is $L(\theta):=\E [\ell (X;\theta)]$. A maximizer of
$L_n(\theta)$ w.r.t. $\theta$ is denoted by $\hat \theta_n=(\hat\mu_n,\hat\sigma_n,\hat\alpha_n,\hat\beta_n)'$. 
We sometimes write scores as a vector
$\ell_{\theta}=(\ell_\mu,\ell_\sigma,\ell_\alpha,\ell_\beta)'$. 
Since we do not know all behaviors of $\ell,\,L$ and $L_n$ at the
boundary $\partial \Theta_M$, our asymptotics are formally done on
arbitrary compact sets $C \subset \Theta_M^\circ$ such that the true parameter
$\theta_0$ is included.  
Our main theorem is as follows.

\begin{theorem}
\label{asymptoticsMLE}
 Let $\hat \theta_n$ be the maximum likelihood estimator based on
 i.i.d. 
 $n$ observations from stable law $(P_\theta:\theta\in \Theta_M)$. 
 Assume that the true parameter $\theta_0$ is in the interior
 $\theta_0\in \Theta_M^\circ$ and prepare an arbitrary compact set $C
 \subset \Theta_M^\circ$ such that $\theta_0\in C$. 
 Then MLE $\hat \theta_n$ restricted on $C$ is consistent and has asymptotic
 normality. In particular we have an expression 
\begin{align}
 \sqrt{n}\,(\hat \theta_n-\theta_0) = I_{\theta_0}^{-1}
 \frac{1}{\sqrt{n}} \sum_{k=1}^n
 \ell_{\theta_0}(X_k)+o_{P_{\theta_0}}(1), 
\end{align}
where $\sqrt{n}\,(\hat \theta_n-\theta_0) \stackrel{d}{\to}
 N(0,I_{\theta_0}^{-1})$ as $n\to\infty$, and $I_{\theta_0}$ is the
 Fisher information matrix. 
\end{theorem}
\noindent
One may think that preparation of a compact set $C$ is a bit
strange. However a similar constraint is imposed on $\hat \theta_n$ in \cite{dumouchel:1973},
since $L_n(\theta)$ in $(B)$ form possibly diverges at the boundary (see
also argument (3) in Section \ref{sec:discussion}).
Notice that with $(M_0)$ parameterization, we can have both consistency 
and asymptotic normality on the whole interior of 
parameter space $\Theta_M$. This is not possible with $(A)$ and $(B)$
expressions since they have discontinuity at $\alpha=1$.  


We give the proof of consistency and that of asymptotic normality
separately. 
For consistency we rely on \cite[Therem 5.7]{vandervaat:2000}
adopted for our purpose, which is 
\begin{theorem}
\label{therem:consistency}
 Suppose that for every $\varepsilon >0$
\begin{align}
& \sup_{\theta \in C}\, |L_n(\theta)-L(\theta) | \stackrel{p}{\to}
 0, \label{consist:firstcondi} \\
& \sup_{\theta: d(\theta,\theta_0)\ge
 \varepsilon,\,\theta \in C}\,L(\theta)<L(\theta_0), \label{consist:secondcondi}
\end{align}
where $d$ is a metric of the parameter space. 
Then any sequence of estimators $\hat \theta_n$ with $L_n(\hat
 \theta_n)\ge L_n(\theta_0) -o_P(1)$ converges in probability to
 $\theta_0$. Here $o_P(1)$ denotes a sequence of r.v.'s converging to
 zero in probability. 
\end{theorem}

\begin{proof}[\it Proof of consistency]
We will check the conditions of Theorem \ref{therem:consistency}. The inequality
 \eqref{consist:secondcondi} 
 is equivalent to the fact: the point $\theta_0\in C$ as a maximizer of  
 continuous function $L(\theta)$ is unique. 
 This is shown by checking the identifiability condition, i.e. $f(\cdot;\theta)\neq
 f(\cdot;\theta')$ for $\theta\neq \theta'$ (see \cite[
 Lemma 5.35]{vandervaat:2000}). However, since for $\theta\neq \theta'$ the
 corresponding ch.f.'s are different, the identifiability follows by the
 uniqueness of the Fourier
 transform. 

An equivalent condition for \eqref{consist:firstcondi} is 
that a set of functions $\theta \mapsto L(x;\theta),\,\theta\in C$ is
 Glivenko-Cantelli. This is implied by the following two conditions:
functions $\theta \mapsto \log f(x;\theta)$ are continuous for
 every $x$ and they are dominated by an integrable envelope
 function (see \cite[p.46]{vandervaat:2000} cf. \cite[Ex. 3.7.3]{vandegeer:2000}). 
 However these conditions are implied by Lemma \ref{lem:ff} and Proposition
 \ref{prop:scores}. 
\end{proof}

For the proof of asymptotic normality, we again rely on
an auxiliary lemma, which is a
combination of Theorem 5.39 and Lemma 7.6 in \cite{vandervaat:2000}.
The lemma is given for a general law $(P_\theta:\theta \in \Theta)$
with density $p_\theta(x)$ and $\Theta \subset \R^k$ is a given
parameter space. 

\begin{lemma}
\label{lemma:asymptnormal}
 For the model $(P_\theta:\theta \in \Theta)$ with density
 $p_\theta(x)$, we assume that the map $\theta \mapsto \sqrt{p_\theta(x)}$
 is continuously differentiable for every $x$. Suppose that the elements
 of the Fisher information matrix $I_\theta$ are well defined and
 continuous in $\theta$. For an inner point $\theta_0$ of $\Theta$, we
 further assume that there exists a measurable function $\dot{\eta}$ with
 $\E_{\theta_0}[\dot{\eta}^2]<\infty$ such that for every $\theta'$ and
 $\theta''$ in a neighborhood of $\theta_0$,  
\begin{align}
\label{lemma:asymptnormal:condi1}
 |\log p_{\theta'}(x)-\log p_{\theta''}(x)| \le
 \dot{\eta}(x)\,\|\theta'-\theta''\|, 
\end{align}
where $\|\cdot\|$ is the Euclidean norm. 
If $I_{\theta_0}$ is nonsingular and $\hat \theta_n$ is consistent for $\theta_0$, then 
\[
 \sqrt{n}\, (\hat \theta_n-\theta_0) = I_{\theta_0}^{-1}\frac{1}{\sqrt{n}}
 \sum_{k=1}^n \ell_{\theta_0}(X_k) + o_{P_{\theta_0}}(1). 
\]
In particular, the sequence $\sqrt{n}\,(\hat \theta_n-\theta_0)$ is
 asymptotically normal with mean zero and covariance matrix
 $I_{\theta_0}^{-1}$. 
\end{lemma}
\noindent 
Notice that first two conditions of Lemma \ref{lemma:asymptnormal} are
sufficient for the ``differentiable in
quadratic mean'' condition in \cite[Theorem 5.39]{vandervaat:2000}, which is the main assertion
of \cite[Lemma 7.6]{vandervaat:2000}. 
\begin{proof}[\it Proof of asymptotic normality]
We check the conditions of Lemma \ref{lemma:asymptnormal} step by step. 
Lemma \ref{lem:ff} implies continuous differentiability of $\sqrt{f(x;\theta)}$ for
 every $x$. We see elements of the Fisher information matrix 
\begin{align}
\label{element:fisheri}
 I_{\theta_i\theta_j}
 = \int \ell_{\theta_i}(x)
 \ell_{\theta_j}(x) / f(x;\theta) dx,\quad i,j=1,2,3,4.
\end{align}
In view of Proposition \ref{prop:scores}, 
 $\ell_{\theta}(x)$ are continuous in $\theta$ for every $x$ and, 
 moreover, by tail conditions of
 $\ell_{\theta_i}$ and $f$, all $4\times 4$ integrands have dominating
 functions which are absolutely integrable. Thus continuity of
 $I_\theta$ in $\theta$ follows from the dominating
 convergence theorem. 
 In order to check
 \eqref{lemma:asymptnormal:condi1}, we apply the mean value theorem to
 obtain 
\begin{align}
 \big| \log p_{\theta'}(x)-\log p_{\theta''}(x) \big| \le \big|
\sum_{i=1}^4 \ell_{\theta_i}(x;\theta_i^\ast) (\theta_i'-\theta_i'')
\big| \le \sum_{i=1}^4 | \ell_{\theta_i}(x;\theta_i^\ast)|\cdot \|\theta'-\theta''\|,
\end{align}
where $\theta_i^\ast$ are vectors between $\theta'$ and $\theta''$. 
In the left side, we have $\E[
(\sum_{i=1}^4 \ell_{\theta_i}(X;\theta_i^\ast) )^2
]<\infty$ since for any $\theta_i^\ast,\theta_j^\ast\in C$,  
 $\E [| \ell_{\theta_i}(X;\theta_i^\ast) \ell_{\theta_j} (X;\theta_j^\ast) |]
 <\infty,\,i,j=1,\ldots,4$ follow from tail behaviors of $\ell_\theta$
 in Proposition \ref{prop:scores}. Thus
 \eqref{lemma:asymptnormal:condi1} follows.

We proceed to the nonsingularity of $I_{\theta_0}$. We take a similar 
approach as in \cite{dumouchel:1973} and prepare linear
 combinations of scores ${\bf a}'\ell_{\theta}=\sum_{j=1}^4 a_j
 \ell_{\theta_j}(x)$ where ${\bf a}=(a_1,a_2,a_3,a_4)' \in \R^4$. 
Since $\E [({\bf a}' \ell_{\theta}(X))^2]$ constitutes a quadratic
 form of $I_\theta$, it suffices to show that $(\ell_{\theta_i})$ are
 linearly independent, namely ${\bf a}' \ell_{\theta}(x)=0$ for all $x$
 if and only if ${\bf a}$ is zero vector. 
 In what follows we assume the former and derive the latter since the opposite direction is
 obvious. 
We use the inversion formula and write
 \begin{align}
\label{linear:conbi:chf}
  {\bf a}' \ell_{\theta}(x)= \frac{1}{2\pi f(x;\theta)}
 \int_{-\infty}^\infty e^{-itx} \sum_{j=1}^4 a_j \varphi_{\theta_j}(t)dt. 
\end{align}
By the uniqueness of the Fourier
 transform, the assumption ${\bf a}' \ell_{\theta}(x)=0$ for all $x$ is 
equivalent to that 
$\sum_{j=1}^4 a_j \varphi_{\theta_j}(t)=0$ for all $t$. 
This implies
\begin{align}
\label{eq:pfmlenons}
 \sum_{j=1}^4 a_j \psi_{\theta_j}(t)=0,\quad \text{for all}\quad t, 
\end{align}
from which we show that ${\bf a}$ is zero vector.
We analyze \eqref{eq:pfmlenons} using expressions of $\psi_{\theta_k}$ in
 Lemma \ref{lem:derivatives:logchfs}. We start with the case
 $\alpha=1,\,\beta\neq 0$. Only $\psi_\alpha$ has the term $|t|^\alpha \log |t|$ and others do not,
 so that $a_3=0$. We focus on $a_2\psi_\sigma+a_4 \psi_\alpha$ and
 collect terms related with $|t|^\alpha$, which is 
\[
 a_2 \psi_\sigma+a_4\psi_\alpha =|t|^\alpha \{-a_2
 \alpha+i\,\sign t\, (a_2\alpha\beta+a_4)\tan (\pi\alpha/2)\}.
\]
Then we should have $a_2=a_4=0$, since both real and imaginary parts
 need to be zero. 
 When $\alpha\neq 1,\,\beta=0$ we have a simpler form
\[
 (\psi_\mu,\psi_\sigma,\psi_\alpha,\psi_\beta)=
 (it,-\alpha|t|^\alpha,-|t|^\alpha \log |t| ,i(t^\alpha-t)\tan(\pi\alpha/2)).
\]
Similarly as before $a_3$ should be zero. Since we could not cancel
 $|t|^\alpha$ of $\psi_\sigma$ by a linear combination of $\psi_\mu$ and
 $\psi_\beta$, it should be $a_1=a_2=0$. 

Next we consider the case $\alpha=1,\,\beta\neq 0$ with expressions
 in \eqref{log:chfa1}. Sine $\psi_\alpha$ includes $\log^2|t|$ and others
 do not, $a_3$ should be zero. We focus on $t \log |t|$ in $\psi_\sigma$
 and $\psi_\beta$ and find $a_4=-\beta a_2$. Then from $\psi_\mu$ and
 $\psi_\sigma$, we have $i t a_1  =(|t|+i(2\beta/\pi) t)a_2$, which is not
 possible unless $a_1=a_2=0$. 
 In the case $\alpha=1,\,\beta=0$, we have 
\[
 (\psi_\mu,\psi_\sigma,\psi_\alpha,\psi_\beta) =(it,-|t|,-|t| \log
 |t|,2/\pi t \log |t|). 
\]
If we look the pairs $(\psi_\mu,\psi_\sigma)$ and
 $(\psi_\alpha,\psi_\beta)$, one element of each pair includes the
 absolute value $|t|$ and the other does not. Thus they are linearly
 independent. 
\end{proof}
Notice that the proof at $\beta=0$ is not implied by asymptotics of
symmetric case since $\beta$ is not estimated there. 

\section{Fisher information around the Cauchy law}
\label{sec:fic}
For confirmation, we numerically examine smoothness of $\varphi_{\theta_i}$ and 
$f_{\theta_i}$ at $\alpha=1$. By using these quantities, we approximate
the Fisher information matrix $I_\theta$ around the Cauchy law
$(\alpha,\beta)=(1,0)$. 
We could not see any discontinuous behaviors of elements in 
$I_\theta$ as $\alpha\to 1,\,\beta\neq 0$ around the Cauchy law, $\beta\in
(-\varepsilon,\varepsilon),\,\varepsilon>0$
, and we observe continuity of all the elements of $I_\theta$ at $(\alpha,\beta)=(1,0)$. 

For convenience the standard $(\mu,\sigma)=(0,1)$ case is considered and
we sometimes write $f(x;\alpha,\beta)$ omitting the location and scale. 
Since $f$ is known to be continuous in $(\alpha,\beta)$ on
$\Theta_M^\circ$, in $I_{\theta_i\theta_j}$ of \eqref{element:fisheri} we
use the Cauchy law for $f$, whereas for $f_\theta$ we take the exact one with the
inversion expression, namely for the integrands we consider
\begin{align}
\label{eq:approx:cauchy}
 \frac{f_{\theta_i}f_{\theta_j}}{f(x;1,0)} = \frac{1}{(2\pi)^2}
 \int_{-\infty}^\infty \int_{-\infty}^\infty \pi (1+x^2)
 e^{-i(s+t)x}\varphi_{\theta_i}(s)\varphi_{\theta_j}(t)ds dt,\quad 1\le
 i,j\le 4. 
\end{align} 
For integration of \eqref{eq:approx:cauchy}, we apply the derivatives of
the Dirac delta $\delta$: $\delta^{(n)}(y)= i^n /(2\pi)\int_{-\infty}^\infty x^n
     e^{ixy} dx$
for $n\in\N$ and their property: for $n$ times continuously
differentiable $h$, $\int_{-\infty}^\infty \delta^{(n)}(t) h(t) dt
=(-1)^n h^{(n)}(0)$. Now applying 
\[
  \int_{-\infty}^\infty e^{-ix (s+t)} \,\pi\, (1+x^2) dx =2\pi^2
  \{ \delta(-(s+t))+\delta''(-(s+t)) \},
\]
we obtain through Fubini's theorem and change of variables that
\begin{align}
\label{ficauchy}
 \tilde I_{\theta_i\theta_j} = \frac{1}{2} \int_{-\infty}^\infty \{
\varphi_{\theta_i}(-t)\varphi_{\theta_j}(t) - \varphi_{\theta_i}'(-t)\varphi_{\theta_j}'(t)
\} dt,
\end{align}
where the exact forms for $\varphi_{\theta_i}$ and $\varphi_{\theta_i}'$
are recovered by Lemma \ref{lem:derivatives:logchfs}.  We also 
evaluate errors of our approximation. Taylor's expression around
the Cauchy density 
yields 
\begin{align*}
 f(x;\alpha,\beta) &= f(x;1,0) +(\alpha-1)f_\alpha
 (x;\alpha^\ast,\beta^\ast) +\beta f_\beta(x;\alpha^\ast,\beta^\ast) 
\end{align*}
where $(\alpha^\ast,\beta^\ast)$ is a value between $(1,0)$ and
$(\alpha,\beta)$, which also may depend on $x$. Then approximation errors for
\eqref{eq:approx:cauchy} are    
\[
 \frac{f_{\theta_i}f_{\theta_j}}{f(x;\alpha,\beta)}-
\frac{f_{\theta_i}f_{\theta_j}}{f(x;1,0)}=
  \left\{
(1-\alpha) \frac{f_\alpha(x;\alpha^\ast,\beta^\ast)}{f(x;1,0)}- \beta\frac{f_\beta(x;\alpha^\ast,\beta^\ast)}{f(x;1,0)}
\right\} \frac{f_{\theta_i}f_{\theta_j}}{f(x;\alpha,\beta)},\quad 1\le
i,j\le 4. 
\]
The integral of the left sides are the errors $I_{\theta_i\theta_j}-\tilde
I_{\theta_i\theta_j}$. 
In view of Lemma \ref{lem:ff} and Proposition \ref{prop:scores}, the
integrands in error terms are uniformly integrable in $(\alpha,\beta)$
close to $(1,0)$. Thus the order of errors is $O(|\alpha-1|)+O(|\beta|)$. 

In Table \ref{table1} we present the exact value of $I_{\theta}$ at
$(\alpha,\beta)=(1,0)$ from \eqref{ficauchy}. 
\begin{table}[httb]
\caption{Fisher information matrix at Cauchy $(\alpha,\beta)=(1,0)$}
  \begin{tabular}{|cccccccccc|} \hline 
     $I_{\alpha\alpha}$ & $I_{\beta\beta}$ & $I_{\sigma\sigma}$ &
   $I_{\mu\mu}$ & $I_{\alpha\beta}$ & $I_{\alpha \sigma}$ & $I_{\alpha
   \mu}$ & $I_{\beta \sigma}$ & $I_{\beta\mu}$ & $I_{\sigma \mu}$  \\ \hline 
 0.859 & 0.348 & 0.5& 0.5 &0 & -0.135 & 0 & 0 & 0.086 & 0\\ \hline 
  \end{tabular} \label{table1}
\end{table}
Our values are consistent with those in \cite{dumouchel:1975}, 
\cite{matsui:takemura:2006} and the values given personally by John
Nolan which are obtained with improvements in the
method of \cite{nolan:1997}. Indeed we could obtain exact values. Let
$\gamma\doteq 0.57722$ be Euler's constant and we have 
$I_{\mu\sigma}=I_{\mu\alpha}=I_{\alpha\beta}=I_{\sigma\beta}=0$,
$I_{\mu\mu}=I_{\sigma\sigma}=0.5$, 
\[
 I_{\alpha\alpha}= \frac{\pi^2}{2} I_{\beta\beta}= \frac{1}{2} \left \{
\frac{\pi^2}{6}+(\gamma+\log 2-1)^2 \right \}\quad \text{and}\quad I_{\sigma\alpha} =-\frac{2}{\pi} I_{\mu\beta}
= \frac{1}{2}(1-\gamma-\log 2). 
\]
\begin{table}[httb]
 \caption{Approximated Fisher informations for $\alpha$ and $\beta$ around Cauchy}
  \begin{tabular}{|l || ccccc|ccccc|} \hline
  &  \multicolumn{5}{|c|}{$\tilde I_{\alpha\alpha}$}
   &\multicolumn{5}{|c|}{$\tilde I_{\beta\beta}$}  \\ \hline 
  $\alpha\setminus\beta$ & $0.1$ & $0.05$ & $0.01$ & $0.001$ & $0$
  & $0.1$ & $0.05$ & $0.01$ & $0.001$ & $0$  \\ \hline
0.95 & 1.096 & 1.087 & 1.084 & 1.084 & 1.084 & 0.392 &0.392 & 0.391& 0.391 & 0.391 \\
0.99 & 0.907& 0.9 & 0.898 & 0.897& 0.897 & 0.357 & 0.356 & 0.356 &0.356 & 0.356\\
0.999 & 0.872 & 0.865 & 0.863 & 0.863& 0.863 & 0.349 &0.349 & 0.349 & 0.349 & 0.349 \\
1 & 0.874 & 0.864 & 0.86 & 0.859& 0.859 & 0.349 & 0.348 & 0.348 & 0.348 & 0.348 \\
1.001 & 0.865 & 0.858 & 0.855 & 0.855 &  0.855 & 0.348  &0.347 & 0.347& 0.347& 0.347 \\
1.01 &0.832 & 0.825 & 0.823 & 0.823 &  0.823 & 0.341 &0.34 & 0.34 & 0.34 & 0.34 \\
1.05 &0.71 & 0.704 & 0.702 & 0.702& 0.702& 0.312 &0.311 & 3.11 & 3.11 & 3.11\\ \hline
  \end{tabular} \label{table2}
\end{table}
Concerning the approximation of $I_\theta$, our numerical study \eqref{ficauchy}
 precisely reflects the theory. 
Namely we numerically confirmed that the Fisher information $I_\theta$ is
 continuous at $\alpha=1,\,\beta\neq0$, although the study is around the
 Cauchy law. Moreover, $I_\theta$ is continuous at
 $(\alpha,\beta)=(0,1)$ as a function of all four parameters. 
We illustrate the elements $I_{\alpha\alpha}$
and $I_{\beta\beta}$ for the range $0.95\le \alpha \le 1.05$ and $-0.1
\le \beta \le 0.1$ in Table \ref{table2}, where values are symmetric about
$\beta=0$ and we omit the case $\beta\le 0$. 
Even when $\alpha=1$ we do not observe large values, which contrasts
 with the Fisher information matrix for $(A)$ form 
by \cite{dumouchel:1973}.  

\section{Discussion and future works}
\label{sec:discussion}
Before we describe preceding and future works we make a remark on 
the parameterization $(B)$. Recall that the form of $(B)$ in
\cite{dumouchel:1973} and that of ours 
are slightly different. To be more precise, only the skewness parameter
$\beta_B'$ in \cite{dumouchel:1973} is different and the others are the same. 
They are connected by 
\[
 \beta_B'=-\beta_B K(\alpha) ,\qquad \text{where} \quad K(\alpha)=\alpha-1+\sign (1-\alpha),
\] 
so that the parameter range of $\beta_B'$ is $|\beta_B'|\le
|K(\alpha)|$. 
In view of the above relation, the asymptotics of MLE in \cite{dumouchel:1973} are
slightly simpler than our version of $(B)$ since the parameter $\alpha$
appears just once in
ch.f. of \cite{dumouchel:1973} 
(cf. \eqref{ch.f.Bform}). Although in practice MLE of $\hat \beta'_B$ may be
more affected by $\hat \alpha$ close to the boundary, the
asymptotic theories of both forms are almost the same. Thus we do not
distinguish two versions in our paper. 

Next we present past researches and discuss about future works. 
In this paper, we analyzed the sores and related functions for $(M_0)$ form on
the interior of the parameter space. Our particular interest is on their
tail behaviors, from which we have derived asymptotics of MLE. As stated
in the introduction, these investigations are sufficient for $(A)$ and
$(B)$ types unless parameters are in the neighborhoods of $\alpha=1$ or the 
boundaries. 

Then next natural questions are what are the behaviors of densities and
scores around the boundaries. These are crucial in
statistical applications. In what follows, we clarify related preceding
researches as possible as we could, focusing on $(A),(B)$ and $(M_0)$ parameterizations. 
Notice that this is not a complete list and we possibly overlook some
references\footnote{
Notice that we confine the list to theories for MLE and calculations of
the related Fisher information matrix. We omit the literature of statistical estimation
methods since they are too many.}. We are welcome for any comments. 
For convenience, the parameters are denoted by
$(\mu_i,\sigma_i,\alpha,\beta_i),\,i=A,B,M_0$ and we take 
the standard cases $(\mu_i,\sigma_i)=(0,1)$. Note that 
three parameterizations are the same in the symmetric case $\beta_i=0$. 
\begin{enumerate}
\item Case $\alpha$ close to $2$.  
In \cite[p.955]{dumouchel:1973} it is pointed out that 
the Fisher information for $\alpha$ (denoted by $I_{\alpha\alpha}$) 
diverges to $\infty$ as $\alpha \to 2$ in $(B)$ forms. 
This fact has also been numerically examined with $(A)$
      form in \cite[Table 1,2A]{dumouchel:1975} and with $(M_0)$ form in \cite[Sec.4]{nolan:2001}.  
These observations are theoretically supported. 
Indeed, the rate of divergence of $I_{\alpha\alpha}$ has been derived in
      the 
      symmetric case (\cite{nagaev:shkolnik:1988}). 
Succeeding the idea
      of \cite{nagaev:shkolnik:1988}, the diverging speed in the non-
      symmetric case has been studied in $(M_0)$ form (see
      \cite{matsui:2005}). 
For the symmetric case, the Fisher information  matrix and MLE around
      $\alpha=2$ have been numerically studied in quite some detail (see \cite{matsui:takemura:2006}). 
\item Case $\alpha$ close to $1$ and $\beta \neq 0$. Although the
      discontinuity at $\alpha=1, \beta \neq 0$ is known for
      parameterizations $(A)$ and $(B)$, 
      the behaviors of scores around $\alpha=1$ have not been analytically
      investigated. Notice that the limit distributions at $\alpha=1$
      exhibit quite different patterns depending on the parameterization (see \cite[p.11,12]{zolotarev}).
      In \cite{dumouchel:1975} the information matrix around $\alpha=1$ has
      numerically been studied, where $I_{\alpha\alpha},\,I_{\beta_A\beta_A}$
      and $I_{\sigma_A \sigma_A}$ showed quite large values. 
\item Case $\alpha$ close to $0$. According to
      \cite[p.955]{dumouchel:1973} in $(B)$ form, the likelihood function
      w.r.t. $\alpha$ and $\mu_B$ has no maximum on $\alpha\in (0,2]$
      and $\mu_B\in \R$. Instead it diverges to $\infty$ as
      $(\alpha,\mu_B) \to (0,x_k)$ where $x_k$ is an observed sample. In
      the symmetric case, the divergence of $I_{\alpha\alpha}$ as
      $\alpha \to 0$ has theoretically been proved (see Theorem 2 in \cite{nagaev:shkolnik:1995}). 
\item Case $\beta$ close to $\pm 1$. It is also pointed out in
      \cite[p.955]{dumouchel:1973} that as $\beta_B\to \pm
      1$
, the
      Fishier information $I_{\beta_B\beta_B}$ approaches
      $\infty$. For $\alpha \in (1,2)$, the tail behaviors as $\beta_B\to
      \pm 1$ are derived in \cite{nagaev:shkolnik:1989}. 
\end{enumerate} 

Now one finds that regardless of the parameterization 
not all boundary cases have been analytically
studied. Here the boundary cases imply that 
$\alpha=0,1,2,\,\beta_i=0,\pm 1,\,i=A,B,M_0$ and their combinations. 
Since stable r.v.'s are assumed in random quantities of many statistical
models, in view of its importance, further investigations are required in
both theory and numerical works. 


Finally, we mention an application in goodness-of-fit tests for stable
laws. Usually these kinds of tests are done with empirical ch.f.'s 
since most stable laws have only closed form density expressions. 
Then asymptotics of empirical ch.f.'s are needed (see \cite{meitanis:2005,
matsui:takemura:2008, meintanis:ngatchou:taufer:2015}). When parameters
are estimated, the weak convergence of empirical ch.f. is assured by
conditions (vi-iv) in \cite{csorgo:1983}. As a by-product of our study
it is shown that 
those conditions are satisfied in $(M_0)$ form. 
Therefore, the theme would be one of our future works.

\appendix
\section{Technical lemmas} 

\subsection{Derivatives of ch.f. and cumulant w.r.t. parameters}
The first and the second derivatives of ch.f. $\varphi(t)$
w.r.t. $\theta$ are given by those of the corresponding cumulant
$\psi(t)=\log \varphi(t)$, namely
$\varphi_{\theta_i}=\psi_{\theta_i}\varphi$ and $\varphi_{\theta_i\theta_j}=(\psi_{\theta_j}
\psi_{\theta_j}+\psi_{\theta_i\theta_j})\varphi$ where
$\theta=(\theta_1,\theta_2,\theta_3,\theta_4)'=(\mu,\sigma,\alpha,\beta)'$. 
In what follows, we present only derivatives for $\psi$ which are
results of straightforward calculations. 
In view of expectations in Lemma \ref{lem:derivatives:logchfs} below one
could observe that $\psi_{\theta_i}(t)$ and $\psi_{\theta_i\theta_j}(t)$ are jointly continuous in
$(t,\theta)\in \R\times \Theta_M^\circ$, so are $\varphi_{\theta_i}$ and
$\varphi_{\theta_i\theta_j}$. 
\begin{lemma}
\label{lem:derivatives:logchfs}
 The first and the second derivatives of $\psi (t)=\log \varphi (t)$ are 
\begin{align*}
 \psi_\mu
 &=it,\,\psi_{\mu\mu}=\psi_{\mu\sigma}=\psi_{\mu\alpha}=\psi_{\mu\beta}=\psi_{\beta\beta}=0, \\
 \psi_\sigma &= -\alpha |t|^\alpha +i t (\alpha |t|^{\alpha-1}-1)\beta
 \tan(\pi \alpha/2), \\
 \psi_\alpha &= -|t|^\alpha \log |t| +it |t|^{\alpha-1} \log |t| \beta
 \tan (\pi\alpha/2) 
 + it (|t|^{\alpha-1}-1)(\pi\beta/2) \cos^{-2}(\pi\alpha/2),\\
 \psi_\beta &= it (|t|^{\alpha-1}-1)\tan(\pi\alpha/2), \\
 \psi_{\sigma\sigma}& = \alpha(\alpha-1)|t|^{\alpha-1}(-|t|+it \beta
 \tan(\pi\alpha/2)), \\
  \psi_{\sigma\alpha}& = |t|^{\alpha-1}(1+ \alpha
 \log|t|)(-|t|+it\beta\tan(\pi\alpha/2))+it
 (\alpha|t|^{\alpha-1}-1)(\beta\pi/2) \cos^{-2}(\pi\alpha/2), \\
 \psi_{\sigma \beta} &= it (\alpha|t|^{\alpha-1}-1)\tan(\pi\alpha/2), \\
 \psi_{\alpha\alpha} &= -|t|^\alpha \log^2|t| +it |t|^{\alpha-1}\log
 |t|\beta \big\{
 \log|t| \tan(\pi\alpha/2)+\pi \cos^{-2}(\pi\alpha/2)
 \big\}, \\
 &\quad + it(|t|^{\alpha-1}-1)(\pi^2\beta/2)
 \cos^{-2}(\pi\alpha/2)\tan(\pi\alpha/2), \\
 \psi_{\alpha\beta} &= it|t|^{\alpha-1} \log |t| \tan(\pi\alpha/2) +it
 (|t|^{\alpha-1}-1)\pi/2 \cos^{-2}(\pi\alpha/2). 
\end{align*}
For $\alpha=1$, the derivatives of $\psi (t)$ 
are as follows.  The quantities
 $\psi_\mu,\psi_{\mu\mu},\psi_{\mu\sigma},\psi_{\mu\alpha},\psi_{\mu\beta}$
 and $\psi_{\beta\beta}$ do not change, and 
\begin{align}
\label{log:chfa1}
\begin{split}
 \psi_\sigma &= -|t|-i (2\beta/\pi) t(1+\log |t|), \quad \psi_\alpha=-|t| \log
 |t| -i(\beta/\pi) t\log^2|t| ,\quad \psi_\beta= -i(2/\pi)t \log |t|, \\
 \psi_{\sigma\sigma}&= -i(2\beta/\pi) t,\quad \psi_{\sigma\alpha}= -
 (1+\log|t|)\big\{ |t| +i (\beta/\pi)t \log |t| \big\},\quad
 \psi_{\sigma\beta}=-i(\pi/2)t (1+\log|t|),  \\
 \psi_{\alpha\alpha}& = i(\pi\beta/3)t  
 \log|t| (1-2/\pi^2 \log^2|t|),\quad \psi_{\alpha\beta}=-i(t/\pi)\log^2|t|.
\end{split}
\end{align}
\end{lemma}
\noindent
The calculations 
for $\alpha=1$ are really complicated. However, one could see the basic idea is 
in the proof of Lemma \ref{lem:dervcauch}, and we omit the details. 

\begin{lemma}
\label{lem:dervcauch}
 The quantities $\varphi_{\theta_i}(t),\,t\varphi_{\theta_i}(t)$ and
 $\varphi_{\theta_i\theta_j}(t),\,i,j=1,\ldots,4$ around $\alpha=1$ are
 respectively bounded by the dominant integrable functions. 
\end{lemma}

\begin{proof}
 We only take up $\varphi_\alpha(t)$ since the proofs for other
 quantities are similar, though some are more complicated. 
 Recall that $\varphi$ includes $e^{-|t|^\alpha}$ and
 $\varphi_\alpha=\psi_\alpha \varphi$. 
 In view of
 $\psi_\alpha,\,\alpha\neq 1$ in Lemma \ref{lem:derivatives:logchfs},
 a dominant function for the term $-|t|^\alpha \log |t|$ is easy, whereas remaining terms include
 $\tan (\pi\alpha/2)$ or $\cos^{-1}(\pi\alpha/2)$ which diverges to $\pm
 \infty$ as $\alpha\to 1$. We focus on the remainder and write 
\[
 A_\alpha(t) = i\beta \tan (\pi\alpha/2) \big \{
 |t|^{\alpha-1} \log |t| +(|t|^{\alpha-1}-1) \pi/\sin(\pi\alpha)
\big \}. 
\]
Observe the following Taylor expansions around $\alpha=1$ with error terms:
\begin{align}
\label{taylorexpansison}
\begin{split}
 |t|^{\alpha-1}\log |t| &= \log |t| +\log^2 |t| \cdot (\alpha-1) + 
 |t|^{\alpha^\ast_1-1} \log^2 |t|\cdot (\alpha-1)^2/2, \\
 |t|^{\alpha-1} -1 &= \log|t|\cdot (\alpha-1) +\log^2|t| \cdot
 (\alpha-1)^2/2 +|t|^{\alpha_2^\ast-1} \log^3|t| \cdot (\alpha-1)^3
 /3!, \\
 \sin(\pi\alpha) &= -\pi (\alpha-1) -\cos (\pi \alpha_3^\ast)
 \pi^3 (\alpha-1)^3/3! ,
\end{split}
\end{align}
where $\alpha_i^\ast,\,i=1,2,3$ are values between $\alpha$ and $1$, and
 $\alpha_i^\ast,\,i=1,2$ depend also on $t$. Applying \eqref{taylorexpansison} to $A_\alpha(t)$, we have
 for $\alpha\neq 1$
\[
 A_\alpha(t)=i\beta \tan(\pi\alpha/2) \big\{\log^2|t| \cdot (\alpha-1)/2 +
 R(t)\cdot (\alpha-1)^2 \big\}, 
\]
where 
\[
 R(t)= c_1 \log |t| +c_2 \log^2|t| \cdot (\alpha-1)+ 
 \big( c_3
 |t|^{\alpha_1^\ast-1} +c_4 |t|^{\alpha_2^\ast-1} \big) \log^3|t|,
\]
and $c_i,\,i=1,\ldots,4$ are constants independent of $t$. 
Now noticing $\tan(\pi\alpha/2)\sim (2/\pi)/(1-\alpha)+O(|\alpha-1|)$, we obtain 
\[
 \lim_{\alpha\to 1} A_\alpha(t) =-i (\beta/\pi) t\log^2|t|,
\]
so that we reach $\psi_\alpha$ in \eqref{log:chfa1}. Looking
 $A_\alpha(t),\,\alpha\neq 1$ above, we observe that 
 $\varphi_\alpha(t)$ is constructed with a linear combination of
 products by 
 $t,\,\log^i|t|,\,i=1,2,3,\,|t|^{\alpha_1^\ast-1}$, 
 $|t|^{\alpha_2^\ast-1}$ and $\varphi(t)$. Since
 $\varphi$ include $e^{-|t|^\alpha}$, we could have an integrable dominant function. 
\end{proof}

\subsection{Tails of density and derivatives in $(B)$ form}
\label{subsec:rela:mb}
In this subsection we briefly explain the tail properties in $(B)$
expression, which are exploited in the main part. As stated in the introduction the form of ch.f. in
$(B)$ type is more convenient than that of $(M_0)$ in analytic point of
view. In \cite{zolotarev} thorough ch.f. various properties of $(B)$ density have been investigated.  
Following \cite{zolotarev}, we treat the standard density
$g(x):=g(x;\alpha,\beta_B)$ of $(B)$ expression where the word 'standard'
implies that location and scale parameters satisfy
$(\mu_B,\sigma_B)=(0,1)$. 
Here the exponent $\alpha$ is uniform $\alpha_M=\alpha_\beta=\alpha$ and
$\beta_B$ denote the 
skewness parameter. 
We are starting with definition of ch.f., 
\begin{align}
\label{ch.f.Bform}
 \varphi_B(t;\alpha,\beta_B)= \left \{
\begin{array}{ll}
 \exp\Big(
-|t|^\alpha \exp( i\frac{\pi}{2}\beta_B K(\alpha)\sign t )
\Big) & \mathrm{if} \quad \alpha\neq 1 \\
 \exp\Big(
 -|t|(\pi/2 +i\beta \log |t| \sign t) 
\Big) & \mathrm{if} \quad \alpha=1, 
\end{array}
\right. 
\end{align}
where $K(\alpha)=\alpha-1+\sign (1-\alpha)$ (see \cite[(2.2.1a),(2.2.1b)]{zolotarev}).   
Then by the inversion formula
\begin{align*}
 g (x;\alpha,\beta_B) = \frac{1}{2\pi} \int_{-\infty}^\infty e^{-itx}
 \varphi_B(t;\alpha,\beta_B) dt = \frac{1}{\pi} \mathrm{Re} \int_0^\infty e^{itx}\varphi_B(t;\alpha,-\beta_B) dt, 
\end{align*}
we reach to the expression \cite[(2.2.1), p.66]{zolotarev}, i.e. 
\begin{align}
g(x;\alpha,\beta_B) & = \left \{
\begin{array}{ll}
\frac{1}{\pi} \mathrm{Re} \int_0^\infty \exp\big(
-itx -t^\alpha 
e^{-i\beta_B \pi/2 K(\alpha)\sign t }
\big) dt
 & \mathrm{if} \quad \alpha\neq 1 \\
\frac{1}{\pi} \mathrm{Re} \int_0^\infty \exp \big(
-itx -\pi/2 t-i\beta_B t \log t
\big) dt
 & \mathrm{if} \quad \alpha=1, 
\end{array}
\right. 
 \label{density:b1} 
\end{align}
which is used in the following analysis. 
In \eqref{density:b1} 
we could see the discontinuity at $\alpha=1$ and thus we restrict parameter space to $\Theta_B=\{\alpha,\beta_B \mid
\alpha \in (0,1)\cup (1,2),\,\beta_B\in(-1,1)\}$. 
As for notations of derivatives, we reuse those of $f$, and
write $g',g'',g_{\theta_i},g_{\theta_i}',g_{\theta_i\theta_j},\varphi_{B,\theta_i}$
and $\varphi_{B,\theta_i\theta_j}$ for $i,j=3,4$, where they mean the
same quantities for those of $f$ except that $g$ and $\varphi_B$ are
always differentiated by $\beta_B$ (not $\beta$). We omit the capital $B$ from
$\beta_B$ for notational simplicity. 

Now we proceed to the main results.
\begin{lemma}
\label{lem:derivg}
Let $(\alpha,\beta_B)\in \Theta_B$. For every $x\in \R$, $g(x;\alpha,\beta_B)$ is twice
 continuity differentiable w.r.t. $(\alpha,\beta_B)$ and $x$. 
 Moreover,
 $g_{\theta_i},\,g_{\theta_i}',\,g',\,g''$ and $g_{\theta_i\theta_j},\,i,j=1,2$ are jointly
 continuous on $\R \times \Theta_B$. 
The tails of the density $g$ and its derivatives satisfy for
 sufficiently large $|x|$,  
\begin{align}
\label{deriveggg}
\begin{split}
\begin{array}{ll}
 g =O(|x|^{-(1+\alpha)}), & g_\beta = O(|x|^{-(1+\alpha)}), \\
 g' = O(|x|^{-(2+\alpha)}),& g'_\beta = O(|x|^{-(2+\alpha)}),\\
g''=O(|x|^{-(3+\alpha)}),& g_{\alpha\alpha} = O(|x|^{-(1+\alpha)} \log^2
 |x|), \\
 g_\alpha = O(|x|^{-(1+\alpha)}\log |x|),& g_{\alpha\beta} = O(|x|^{-(1+\alpha)}),\\
g'_\alpha =
 O(|x|^{-(2+\alpha)}\log |x|),& g_{\beta\beta} = O(|x|^{-(1+\alpha)}). 
\end{array}
\end{split}
\end{align}
\end{lemma}

\begin{proof}
 For the first, second and joint differentiabilities, the proof is done exactly the same
 way as that in Lemma \ref{lem:ff} and we omit it. 
We show the tail bounds separately for $\alpha<1$ and $\alpha>1$. \\
{\bf Case $\alpha<1$. } 
First we consider the case $x>0$. 
We use another representation of the density in $(B)$ form 
\[
 g(x;\alpha,\beta_B) = \frac{1}{\pi} \mathrm{Im}\int_0^\infty \exp\big \{
-xu-u^\alpha e^{
-i\pi\alpha/2 (1+\beta_B)}
\big \} du,\quad \alpha<1,\, x>0,  
\]
which is obtained by considering a contour integral on the complex plane (see 
\cite[(2.2.8) in Theorem 2.2.1]{zolotarev}). 
We only explain those of $g''$ and
 $g_\alpha$ since the other cases are similar. A straightforward calculation yields 
\begin{align*}
 g''(x;\alpha,\beta_B) &= \frac{1}{\pi x^3} \mathrm{Im} \int_0^\infty v^2
 \exp\big \{
-v-x^{-\alpha}v^\alpha e^{
-i\pi\alpha/2 (1+\beta_B)}
\big \} dv \\
&= \frac{1}{\pi x^3} \mathrm{Im} \sum_{k=0}^\infty 
\int_0^\infty \frac{\{-v^\alpha x^{-\alpha}e^{
 -i\pi\alpha/2 (1+\beta_B)}\}^k}{k!} v^2 e^{-v}dv,
\end{align*}
where we use Fubini's theorem for exchange of the improper integral and
 the infinite sum, which is possible for sufficiently large $x$. Since we need to take
 the imaginary part, the integral of term $k=1$ is dominant and we get
 the result. Moreover,
\begin{align*}
 g_\alpha(x;\alpha,\beta_B) &= \frac{1}{\pi} \mathrm{Im} \int_0^\infty
\exp\big\{
-v-v^\alpha x^{-\alpha} e^{-i\pi\alpha/2 (1+\beta_B)}
\big\}\\
&\quad \times \big\{
x^{-\alpha-1}v^{\alpha}(\log x-\log v) +i
 \pi/2(1+\beta_B)x^{-\alpha-1}v^{-\alpha}\big\}
e^{
-i\pi\alpha/2 (1+\beta_B)} dv  \\
&= O(x^{-(1+\alpha)}\log x)
\end{align*}
as $x\to \infty$ by DCT. 
The case $x< 0$ is derived from 
those for $x>0$ by the relation in \cite[p.65]{zolotarev}, 
\begin{align}
\label{relation:B}
  g(-x;\alpha,\beta_B)=g(x;\alpha,-\beta_B),\quad x\ge 0\quad \text{and}\quad \beta_B
 \ge 0. 
\end{align}
\noindent
{\bf Case $\alpha>1$.}
For $x>0$
we use the relation \cite[p.94, (2.5.5)]{zolotarev},
\begin{align}
\label{g:origan-tail}
 g(x;\alpha,\beta_B)=x^{-1-\alpha}g(x^{-\alpha};\alpha',\beta_B' ), 
\end{align}
where $\alpha'=1/\alpha$ and $\beta_B'=1-(2-\alpha)(1+\beta_B)$. 
We differentiate both sides and represent partial derivatives for
 $\alpha>1$ by combinations of derivatives of $x^{-1-\alpha}$ and 
 and $g(x^{-\alpha};\alpha',\beta_B')$ with $\alpha'<1$. 
 The partial derivatives of $g$ around the origin for $\alpha<1$ converge to
 constants (could be zero) as $x\to 0$, which are observed by the direct
 differentiation of \eqref{density:b1} under the integral sign. Therefore by
 letting $x\to\infty$ (so that $x^{-\alpha}\to 0$), in the right side of
 \eqref{g:origan-tail}, the tail bounds for $x>0$ are derived. 
 These bounds depend on $\alpha$ in the same manner as in case
 $\alpha<1$. The case $x<0$ follows again by \eqref{relation:B}. 
\end{proof}
Notice that bounds of \eqref{deriveggg} are not always exact and we
could obtain better ones depending on parameters. 
\begin{remark}
 For the proof of Lemma \ref{lem:derivg} we
 could alternatively exploit series expansions (2.4.6), (2.4.8), (2.5.1)
 and (2.5.4) in \cite{zolotarev}, some of which are obtained by \cite{bergstrom}. 
 In fact, the approach in \cite{dumouchel:1973} depends on series expansions by \cite{bergstrom}. 
 We have checked that term-wise differentiations are possible for these expansions and could obtain the
 same results. The derivation of the expansions requires a systematic treatment of the complex contour
 integral (see Sec.2.4,2.5 in \cite{zolotarev}) which is a considerable
 burden for readers. We derive the results directly from the expressions
 of the inversion formula. 
\end{remark}

\subsection{Expressions for derivatives of $f$ with those of $g$}
Next we see the relation of $f$ and $g$. Again we take the standard
density for $g$. The parameter $\alpha$ is
uniform and the skewness parameters $\beta$ and $\beta_B$ are linked by 
\begin{align}
\label{relation:betaB}
 \beta_B=\frac{\arctan \beta \tan \frac{\pi\alpha}{2}}{\pi K(\alpha)/2},   
\end{align}      
so that $\beta\neq \pm 1 \Leftrightarrow \beta_B \neq \pm 1$. First we
express $f$ with $g$ comparing \eqref{fdensity:m} and \eqref{density:b1}.
In \eqref{fdensity:m} we change variables $t\to \gamma_{\alpha,\beta} s$ where
$$ \gamma_{\alpha,\beta}=\cos^{1/\alpha}(\pi K(\alpha)\beta_B) = 
(1+\beta^2\tan^2(\pi \alpha/2))^{-1/(2\alpha)}.$$
and obtain 
\begin{align*}
 f(x;\mu,\sigma,\alpha,\beta) 
&= \frac{\gamma_{\alpha,\beta}}{\sigma}
\frac{1}{\pi} \mathrm{Re} \int_0^\infty e^{it
\frac{\gamma_{\alpha,\beta}}{\sigma}
(x-\mu+\sigma\beta \tan(\pi\alpha/2))-|t|^\alpha\exp(-i\pi K(\alpha)\beta_B) } dt \\
&= \frac{\gamma_{\alpha,\beta}}{\sigma \pi} \, g \big(
 x^\ast,\alpha,\beta_B 
\big),\quad \alpha\neq 1,\beta\in(-1,1),
\end{align*}
where
\begin{align}
\label{def:xast}
  x^\ast  =\frac{\gamma_{\alpha,\beta}}{\sigma}\Big ( x-\mu+\sigma \beta \tan
 (\pi\alpha/2)\Big ) . 
\end{align}
Since $g$ is infinitely differentiable with $x,\alpha,\beta_B$ on
$\R\times\Theta_B$, we can express derivatives of $f$ by those of $g$. 

In what follows $f$ and its derivatives are evaluated at $(x,\theta)$
and those of $g$ are evaluated at $(x^\ast,\alpha,\beta_B)$ where
$\beta_B$ is given by \eqref{relation:betaB} and \eqref{def:xast}. 
\begin{lemma}
\label{lem:a1}
Let $\theta\in \Theta_M^\circ \cap \{\alpha  \neq 1\}$ so that $(\alpha,\beta_B)\in \Theta_B$. 
The expression for derivatives of $f$ by those of $g$ are 
\begin{align}
\label{derivf:g1}
 f_\mu &= -f'= c_{\mu} g', \nonumber \\
 f_\sigma &=c_{\sigma,1}g+ (c_{\sigma,2} + c_{\sigma,3}x)g', \nonumber\\
 f_\alpha &=c_{\alpha,1}g+ (c_{\alpha,2} +
  c_{\alpha,3}x)g'+c_{\alpha,4}g_\alpha+c_{\alpha,5}g_\beta,  \nonumber\\
 f_\beta &=c_{\beta,1}g+(c_{\beta,2}+c_{\beta,3}x)g'+c_{\beta,4}g_\beta,\nonumber\\
 f_{\mu\mu} &= -f_\mu'= c_{\mu \mu} g'', \nonumber\\
 f_{\mu \sigma} &= -f_\sigma'= c_{\mu \sigma,1} g'+ (c_{\mu
 \sigma,2}+c_{\mu \sigma,3}x)g'', \nonumber\\
 f_{\mu \alpha} &= -f_\alpha'=c_{\mu \alpha,1}g' +(c_{\mu
 \alpha,2}+c_{\mu \alpha,3}x)g''+
 c_{\mu \alpha,4}g'_\alpha + c_{\mu \alpha,5}g'_\beta, \nonumber\\
 f_{\mu \beta}  &= -f_\beta'=
 c_{\mu \beta,1}g'+(c_{\mu \beta,2}+c_{\mu \beta,3}x)g''+ c_{\mu \beta,4}g'_\beta,\nonumber\\
 f_{\sigma\sigma} &=c_{\sigma\sigma,1}g+ (c_{\sigma\sigma,2} +
  c_{\sigma\sigma,3}x)g'+(c_{\sigma\sigma,4}+c_{\sigma\sigma,5}x +
 c_{\sigma\sigma,6}x^2) g'',  \nonumber\\
 f_{\sigma\alpha} &=c_{\sigma\alpha,1}g+ (c_{\sigma\alpha,2} +
  c_{\sigma\alpha,3}x)g'+(c_{\sigma\alpha,4}+c_{\sigma\alpha,5}x +
 c_{\sigma\alpha,6}x^2) g''  \\
 &\quad + c_{\sigma\alpha,7}g_\alpha+ (c_{\sigma\alpha,8} +
  c_{\sigma\alpha,9}x)g_\alpha'+ c_{\sigma\alpha,10}g_\beta+
 (c_{\sigma\alpha,11}+c_{\sigma\alpha,12}x) g_\beta', \nonumber\\
 f_{\sigma \beta} &= c_{\sigma\beta,1}g+ (c_{\sigma\beta,2} +
  c_{\sigma\beta,3}x)g'+(c_{\sigma\beta,4}+c_{\sigma\beta,5}x +
 c_{\sigma\beta,6}x^2) g'' + c_{\sigma\beta,7}g_\beta +
 c_{\sigma\beta,8} g_\beta', \nonumber \\
 f_{\alpha\alpha} &= c_{\alpha\alpha,1}g+ (c_{\alpha\alpha,2} +
  c_{\alpha\alpha,3}x)g'+(c_{\alpha\alpha,4}+c_{\alpha\alpha,5}x +
 c_{\alpha\alpha,6}x^2) g'' + c_{\alpha\alpha,7}g_\alpha \nonumber\\
 &\quad + (c_{\alpha\alpha,8} +
  c_{\alpha\alpha,9}x)g_\alpha'+ c_{\alpha\alpha,10}g_\beta+
 (c_{\alpha\alpha,11}+c_{\alpha\alpha,12}x) g_\beta' \nonumber\\
 &\quad + c_{\alpha\alpha,13}x g_{\alpha\alpha} +
 c_{\alpha\alpha,14}g_{\alpha\beta} + c_{\alpha\alpha,15}g_{\beta\beta}
 \nonumber \\
 f_{\alpha \beta} &= c_{\alpha\beta,1}g+ (c_{\alpha\beta,2} +
  c_{\alpha\beta,3}x)g'+(c_{\alpha\beta,4}+c_{\alpha\beta,5}x +
 c_{\alpha\beta,6}x^2) g''+ c_{\alpha\beta,7} g_{\alpha}\nonumber \\
 &\quad +
 (c_{\alpha\beta,8}+c_{\alpha\beta,9}x )g_{\alpha}' +
 c_{\alpha\beta,10}g_{\beta}+  (c_{\alpha\beta,11}+c_{\alpha\beta,12}x
 )g_{\beta}'\nonumber \\
 &\quad + c_{\alpha\beta,13}x g_{\alpha\beta} +
 c_{\alpha\beta,14}g_{\beta\beta},\nonumber \\
 f_{\beta \beta} & = c_{\beta\beta,1}g+ (c_{\beta\beta,2} +
  c_{\beta\beta,3}x)g'+(c_{\beta\beta,4}+c_{\beta\beta,5}x +
 c_{\beta\beta,6}x^2) g'' \nonumber\\
 &\quad + c_{\beta\beta,7}g_{\beta} +
 (c_{\beta\beta,8}+c_{\beta\beta,9}x)g_{\beta}' +
 c_{\beta\beta,10}g_{\beta\beta}. \nonumber
\end{align}
\end{lemma}

\noindent {\bf Acknowledgment:} 
I would like to thank Prof. John Nolan for teaching related references 
and for showing numerical values of the Fisher information matrix at the Cauchy distribution. 
\end{document}